\documentclass[11pt]{amsart}
\usepackage{amsfonts}
\usepackage{bbm}
\usepackage{amsfonts,amssymb,amsmath,amsthm}
\usepackage{url}
\usepackage{enumerate}
\usepackage{bbm}
\usepackage[all]{xy}
\usepackage[pdftex, colorlinks, citecolor=red, backref=page]{hyperref}
\usepackage{color,xcolor}

\newtheorem{theorem}{Theorem}[section]
\newtheorem{lemma}[theorem]{Lemma}
\newtheorem{definition}[theorem]{Definition}
\newtheorem{notion}[theorem]{}

\newtheorem{proposition}[theorem]{Proposition}
\newtheorem{corollary}[theorem]{Corollary}
\theoremstyle{definition}
\newtheorem{remark}[theorem]{Remark}

\author{Qingnan An}
\address{School of Mathematics and Statistics, Northeast Normal University, Changchun, {\rm130024}, China}
\email{qingnanan1024@outlook.com}

\author{Zhichao Liu}
\address{School of Mathematical Sciences,
Dalian University of Technology,
Dalian, {\rm116024}, China}
\email{lzc.12@outlook.com}


\begin{document}

\title[Unital absorbing extensions] {On unital absorbing extensions of C$^*$-algebras of stable rank one and real rank zero}

\begin{abstract}
Suppose that $A,B$ are nuclear, separable ${\rm C}^*$-algebras of stable rank one and real rank zero, $A$ is unital simple, $B$ is stable and $({\rm K}_0(B),{\rm K}_0^+(B))$  is weakly unperforated in the sense of Elliott \cite{Ell}. We show that any unital extension with trivial index maps of $A$ by $B$ is absorbing.

\end{abstract}

\maketitle
\section*{Introduction}

The theory of ${\rm C}^*$-algebra extensions began with the work of Brown, Douglas and Fillmore on classifying essentially normal operators \cite{BDF}. But soon it develops into  a theory related to many fields of mathematics. Absorbing extensions play an important role in the theory of ${\rm C}^*$-algebras. The noncommutative  Weyl-von Neumann theorem of Voiculescu \cite{V} shows that for any separable ${\rm C}^*$-algebra $A$, every (unital) essential extension of $A$ by $\mathcal{K}$ is (unital) absorbing.
For the more general separable case, the essential extensions of  $A$ by $B\otimes\mathcal{K}$, Thomsen shows that there always exists a (unital) trival extension which is (unital) absorbing \cite{T}. 
It is a non-trivial task to decide that under which conditions any essential extension of separable ${\rm C}^*$-algebras is absorbing.

 In this paper, we consider the ${\rm C}^*$-algebra extensions of $A$ by $B$.
We say an extension $e:0\to B\to E\to A\to 0$ has $trivial$ $index$ $maps$, if  the induced six-term exact sequence
$$
\xymatrixcolsep{2pc}
\xymatrix{
{\,\,\mathrm{K}_0(B)\,\,} \ar[r]^-{}
& {\,\,\mathrm{K}_0(E)\,\,}\ar[r]^-{}
& {\,\,\mathrm{K}_0(A)\,\,} \ar[d]^-{\delta_0}
 \\
{\,\,\mathrm{K}_1(A)\,\,} \ar[u]^-{\delta_1}
& {\,\,\mathrm{K}_1(E)\,\,}\ar[l]^-{}
& {\,\,\mathrm{K}_1(B)\,\,}\ar[l]^-{} }
$$
satisfies that both $\delta_0$ and $\delta_1$ are the zero maps.
Using  the intrinsic characterizations established by Elliott and Kucerovsky \cite{EK} and techniques for stable C*-algebras, we prove the following theorem.
\begin{theorem}
Let $A,B$ be nuclear, separable ${\rm C}^*$-algebras of stable rank one and real rank zero. Suppose that $A$ is simple,  $B$ is stable and $({\rm K}_0(B),{\rm K}_0^+(B))$  is weakly unperforated. Let $e$ be a unital extension with trivial index maps:
$$
0\to B\to E\to A\to 0.
$$
Then $e$ is absorbing.
\end{theorem}

We apply this result to the stably strong Ext-groups and obtain the correspondence from the isomorphic classes of  extensions of ${\rm K}_*$-groups to subgroup of strong  Ext-groups.  





\section{Preliminaries}

  Let $A$ be a unital $\mathrm{C}^*$-algebra. $A$ is said to have stable rank one, written $sr(A)=1$, if the set of invertible elements of $A$ is dense. $A$ is said to have real rank zero, written $rr(A)=0$, if the set of invertible self-adjoint elements is dense in the set $A_{sa}$ of self-adjoint elements of $A$. If $A$ is not unital, let us denote the minimal unitization of $A$ by $\widetilde{A}$. A non-unital $\mathrm{C}^*$-algebra is said to have stable rank one (or real rank zero) if its unitization has stable rank one (or real rank zero).

  Denote $\mathcal{P}(A)$ the set consisting  of all the projections in $A$.
  Let $p,$ $q\in \mathcal{P}(A)$.
  One says that $p$ is $Murray$--$von$ $Neumann$ $equivalent$ to $q$ in $A$ and writes $p\sim q$ if there exists $x\in A$ such that $x^*x=p$ and $xx^*=q$. We will write $p\preceq q$ if $p$ is equivalent to some subprojection of $q$; we will write $p\leq q$ if $p$ is a subprojection of $q$. Denote $\mathcal{K}$ the ${\rm C}^*$-algebra of the compact operators on an infinite-dimensional separable Hilbert space. The classes of  projections $p,q\in A\otimes \mathcal{K}$ in $\mathrm{K}_0(A)$ will be denoted by $[p],[q]$, and we say $[p]\leq [q]$, if $p\preceq q$ in $A\otimes \mathcal{K}$.

$A$ is said to have cancellation of projections, if for any projections $p,q,e,f\in A$ with $pe=0$, $qf=0$, $e\sim f$, and $p+e\sim q+f$, then $p\sim q$. $A$ has cancellation of projections if and only if $p\sim q$ implies that there exists a unitary $u\in \widetilde{A}$ such that $u^*pu=q$ (see \cite[3.1]{L}). Every unital $\mathrm{C}^*$-algebra of stable rank one has cancellation of projections, hence, any two projections with the same ${\rm K}$-theory generate the same ideal. 
 We will use $I_p$  to represent the ideal generated by $p$ in $A$.

A ${\rm C}^*$-algebra $A$ is $stable$ if  it is isomorphic to its tensor product with the $\mathrm{C}^*$-algebra $\mathcal{K}$, i.e., $A\cong A\otimes \mathcal{K}$.

A ${\rm C}^*$-subalgebra $B$ of a ${\rm C}^*$-algebra $A$ is said to be $full$ if it is not contained in any proper two-sided closed ideal of $A$. Denote ${\mathcal M}(A)$ the multiplier algebra of $A$.
Given a projection $p$ in ${\mathcal M}(A)$, $pAp$ is a hereditary subalgebra of $A$,
 which will be called a corner. A projection $p$ in ${\mathcal M}(A)$ is full if  the corner $pAp$ is full in $A$.

\begin{definition}\rm
Let $A$, $B$ be ${\rm C}^*$-algebras and
$$e\,:\, 0 \to B \to E \xrightarrow{\pi} A \to 0 $$
be an extension of $A$ by $B$ with Busby invariant $\tau:\,A\to {\mathcal M}(B)/B$.
We say the extension $e$ is $essential$, if $\tau$ is injective.
When $A$ is unital, we say  $e$ is $unital$, if $\tau$ is unital;
we say a unital extension $e$ is $unital$ $trivial$, if there is
a unital $*$-homomorphism $\rho\, : A \to E$ such that $\pi\circ\rho = id_A$.

Note that if $A$ is unital simple and $B\neq 0$, any unital extension $e$ of $A$ by $B$ is essential.
\end{definition}
\begin{definition}\rm
Let
$e_i\,:\, 0 \to B \to E_i \to A \to 0 $
be two
extensions with Busby invariants $\tau_i$ for $i = 1, 2$. Then $e_1$ and $e_2$ are called $strongly$ $unitarily$ $equivalent$, denoted by $e_1
\sim_s e_2$,
if there exists a unitary $u \in {\mathcal M}(B)$ such that $\tau_2(a) = \pi(u)\tau_1(a)\pi(u)^*$ for all $a \in A$,
where $\pi:\,{\mathcal M}(B)\to {\mathcal M}(B)/B.$

Assume $B$ is stable. Fix an isomorphism of $\mathcal{K}$ with $M_2(\mathcal{K})$; this isomorphism induces an isomorphism $B \cong M_2(B)$ and hence isomorphisms ${\mathcal M}(B) \cong M_2({\mathcal M}(B))$ and  ${\mathcal M}(B)/B \cong M_2({\mathcal M}(B)/B)$. These isomorphisms are uniquely determined up to unitary equivalence.

Let $s_1,s_2$ be two isometries in ${\mathcal M}(B\otimes\mathcal{K})$ with $s_1s_1^*+s_2s_2^*=1$.
Define the addition $e_1 \oplus e_2$ to be the extension of $A$ by $B$ with Busby invariant
$$
\tau_1\oplus\tau_2:=s_1\tau_1s_1^*+s_2\tau_2s_2^*.
$$

We say that a unital extension $e$ with Busby invariant $\tau$ is $absorbing$ if it is strongly unitarily equivalent to its sum with any  unital trivial extension, i.e.,  $e \sim_s e\oplus \sigma$ for any unital trivial extension $\sigma$.

\end{definition}

\begin{definition}\rm (\cite{EK})
  Let $E$ be a $\mathrm{C}^*$-algebra and $B$ be a closed two-sided ideal of $E$. Let us say that $E$ is $purely$ $large$ with respect to $B$ if for every element $c$ in $E\backslash B$, the $\mathrm{C}^*$-algebra $\overline{c B c^*}$ (the intersection with $B$ of the hereditary sub-$\mathrm{C}^*$-algebra of $E$ generated by $cc^*$) contains a sub-$\mathrm{C}^*$-algebra which is stable and is full in $B$.

Let
$
e\,:\,0 \rightarrow B \rightarrow E \rightarrow A \rightarrow 0
$
be an extension of $A$ by $B$, we say that $e$ is $purely\,\,large$ if the $\mathrm{C}^*$-algebra $E$ is purely large with respect to the image of $B$.
\end{definition}


The following is the main result in \cite{EK} under the nuclear setting.
\begin{theorem}{\rm (}\cite[Theorem 6]{EK}{\rm )}\label{ek thm}
  Let $A$ and $B$ be nuclear, separable $\mathrm{C}^*$-algebras, with $B$ stable and $A$ unital. A unital extension of  $A$ by $B$ is absorbing if and only if it is purely large.
\end{theorem}

 We will use the following theorem frequently.
\begin{theorem}{\rm (}\cite[Theorem 2.8]{B}{\rm )}.\label{herideal}
  If $B$ is a full hereditary ${\rm C}^*$-subalgebra of $A$ and if each of $A$ and $B$ has a strictly positive element, then $B\otimes\mathcal{K}$ is canonically isomorphic to $A\otimes\mathcal{K}$.
\end{theorem}
\begin{remark}\label{iota inj}
Let $A$ be a separable ${\rm C}^*$-algebra of stable rank one and $p$ be a full projection in $\mathcal{M}(A)$. Denote $\iota:\,pAp\to A$ the natural embedding map.
By Theorem \ref{herideal}, $pAp\otimes\mathcal{K}$ is isomorphic to $A\otimes\mathcal{K}$, and  the isomorphism is induced by a partial isometry $v\in {\mathcal M}(A\otimes\mathcal{K})$ which takes $x$ to $v^*xv$. 

We note that the inclusion $\iota$ induces the same order on $
{\rm K}_0$ with $v^*(\cdot)v$. For any projection $q\in pAp\otimes\mathcal{K}$, we identify
$(\iota\otimes id_\mathcal{K})(q)$ with $q$.
Note that $v^*q\in A\otimes\mathcal{K}$ satisfies
$$
v^*q(v^*q)^*=v^*qv\quad{\rm and}\quad (v^*q)^*v^*q=q,
$$
which implies that ${\rm K}_0(\iota)([q])=[v^*qv]$. (A more  generalised result can be found in \cite[Proposition 5.1]{GL}.)

\end{remark}

For any ${\rm C}^*$-algebra $A$, the positive cone of ${\rm K}_0$ and the scale are given by
$${\rm K}_0^+(A)= \{[p] :\, p \in \mathcal{P}(A \otimes \mathcal{K})\},\quad \Sigma\,A = \{[p] :\, p \in \mathcal{P}(A)\}.$$
 Now we list some results for stable algebras due to R\o rdam.
\begin{proposition} {\rm (}\cite[Proposition 3.1]{R1}{\rm )}\label{rodam p1}
Let $A$ be a ${\rm C}^*$-algebra with the cancellation property and with a
countable approximate unit consisting of projections. Then $A$ is stable if and only if $\Sigma\, A= {\rm K}_0^+(A).$
\end{proposition}

\begin{definition}\label{scale def}\rm  (\cite[Definition 3.2]{R1})
A triple $\left(G, G^{+}, \Sigma\right)$ will be called a scaled, ordered abelian group if $\left(G, G^{+}\right)$ is an ordered abelian group, and $\Sigma$ is an upper directed, hereditary, full subset of $G^{+}$, i.e.,

(i) $\forall\, x_1, x_2 \in \Sigma,\, \exists\, x \in \Sigma: x_1 \leq x, x_2 \leq x$,

(ii) $\forall\, x \in G^{+},\, \forall\, y \in \Sigma: x \leq y \Rightarrow x \in \Sigma$,

(iii) $\forall \, x \in G^{+},\, \exists\, y \in \Sigma,\, \exists\, k \in \mathbb{N}: x \leq k y$.

Let $\left(G, G^{+}\right)$ be an ordered abelian group, and let $\Sigma_1$ and $\Sigma_2$ be upper directed, hereditary, full subsets of $G^{+}$. Define $\Sigma_1\, \hat{+} \,\Sigma_2$ to be the set of all elements $x \in G^{+}$ for which there exist $x_1 \in \Sigma_1$ and $x_2 \in \Sigma_2$ with $x \leq x_1+x_2$. Observe that $\Sigma_1\, \hat{+} \,\Sigma_2$ is an upper directed, hereditary, full subset of $G^{+}$.

Denote the $k$-fold sum $\Sigma\, \hat{+} \,\Sigma \hat{+} \cdots \hat{+} \Sigma$ by $k \cdot \Sigma$. Using that $\Sigma$ is upper directed,  we see that $y \in k \cdot \Sigma$ if and only if $0 \leq y \leq k x$ for some $x \in \Sigma$.
\end{definition}

If $A$ is  a ${\rm C}^*$-algebra of stable rank one and has an approximate unit consisting of projections, then  $\left({\rm K}_0(A), {\rm K}^{+}_0(A), \Sigma \,A\right)$ is a scaled, ordered abelian group and 
$$
\left({\rm K}_0\left(M_k(A)\right), {\rm K}_0^{+}\left(M_k(A)\right), \Sigma\left(M_k(A)\right)\right) \cong\left({\rm K}_0(A), {\rm K}^{+}_0(A), k\, \hat{\cdot}\, \Sigma\, A\right).
$$
\begin{definition}\rm {\rm (}\cite{Ell},\cite[8.1]{Goo2}{\rm )}\label{weakly un}
An ordered abelian group $\left(G, G_{+}\right)$ is called weakly unperforated if it satisfies both of the following two conditions.

(i) Given $x \in G$ and $m \in \mathbb{N}$ such that $m x \in G_{+}$, there exists $t \in \operatorname{tor}(G)$ such that $x+t \in G_{+}$ and $m t=0$.

(ii) Given $y \in G_{+}, t \in \operatorname{tor}(G)$, and $n \in \mathbb{N}$ such that $n y+t \in G_{+}$, then $y \pm t \in G_{+}$.
\end{definition}
Now we prove the following result for the nonsimple ordered group case, which generalizes \cite[Proposition 3.3]{R1}.
\begin{proposition}\label{rodam p2}
Let $(G, G_+,\Sigma)$ be a weakly unperforated, scaled, ordered, abelian
group, and suppose that $n\, \,\hat{\cdot} \,\Sigma = G_+$
for some $n\in\mathbb{N}$. Then $\Sigma= G_+$.
\end{proposition}
\begin{proof}
Let $g\in G_{+}$, then $(n+1)g\in G_+$. Since $n \,\hat{\cdot}\, \Sigma=G_{+}$, there is an element $x \in \Sigma$ with $n x \geq (n+1)g$. From $n(x-g) \in G_+$ and Definition \ref{weakly un} (i), there exists $t\in {\rm tor}(G)$ such that $x-g+t\in G^+$ and $nt=0$.  Then $x+t\in G_+$, and hence,
$$
n(x+t)+t=(n-1)x+x+t\in G_+.
$$
By  Definition \ref{weakly un} (ii), we have $x+2t\in G^+$. Then,
$$
(nx-(n+1)g)+(x+2t)+nt=(n+1)(x-g+t)+t\in G^+.
$$
Using  Definition \ref{weakly un} (ii) again, we have
$$
x-g=x-g+t-t\in G^+.
$$
By the hereditary property of $\Sigma$, we obtain $g \in \Sigma$. Thus $\Sigma=G^{+}$.

\end{proof}
The following  is a straight result from above.
\begin{corollary} \label{cor K0}
  Let $A$ be a separable ${\rm C}^*$-algebra of stable rank one and real rank zero and with a weakly unperforated $({\rm K}_0(A),{\rm K}_0^+(A))$. If  $M_n(A)$ is stable for some integer $n\geq 1$, then $A$ is stable.
\end{corollary}
\begin{proof}
If  $M_n(A)$ is stable for some integer $n\geq 1$, with the fact that
$$
\left({\rm K}_0\left(M_n(A)\right), {\rm K}_0^{+}\left(M_n(A)\right), \Sigma\left(M_n(A)\right)\right) \cong\left({\rm K}_0(A), {\rm K}^{+}_0(A), n\, \hat{\cdot}\, \Sigma\, A\right),
$$
by Proposition \ref{rodam p1}, we have
$
n\, \hat{\cdot}\, \Sigma\, A={\rm K}^{+}_0(A).
$
Then Proposition \ref{rodam p2} implies
$
\Sigma\, A={\rm K}^{+}_0(A).
$
By Proposition \ref{rodam p1} again,
$A$ is stable.

\end{proof}
\begin{remark}
We claim that without the assumption of $A$ has a countable approximate unit consisting of projections, $({\rm K}_0(A),{\rm K}_0^+(A),\Sigma\,A)$ is not necessary a scaled, ordered, abelian group. We put a short example here.

Take
$$A=\{f\in M_2(C[0,1]):\,
f(0)=\fontsize{8pt}{5pt}\selectfont\left(\!\!\begin{array}{cc}
\lambda&0\\[1.5 mm]
0&\lambda
\end{array}\!\!\right)~~{\rm and}~~
f(1)=\fontsize{8pt}{5pt}\selectfont\left(\!\!\begin{array}{cc}
\mu&0\\[1.5 mm]
0&0
\end{array}\!\!\right), ~~{\rm }~~ \lambda,\mu\in\mathbb{C}\}.$$
Then we have the following exact sequence,
$$
0\to M_2(C_0(0,1))\xrightarrow{\iota}  A\xrightarrow{\pi} \mathbb{C}\oplus\mathbb{C}\to 0,
$$
where $\iota$ is the natural embedding and $\pi(f)=(\lambda,\mu)$ for $f\in A$.
The induced well-known six-term exact sequence goes as follows:
$$
0\to {\rm K}_0(A)\to\mathbb{Z}\oplus \mathbb{Z}
\xrightarrow{(2,-1)} \mathbb{Z}\to {\rm K}_1(A)\to 0,
$$
in which ${\rm K}_0(M_2(C_0(0,1)))=0$ and ${\rm K}_1(\mathbb{C}\oplus\mathbb{C})=0$.
So we obtain that
$${\rm K}_0(A)=\{(a,b)\in \mathbb{Z}\oplus \mathbb{Z}:\,
2a-b=0\}\cong \mathbb{Z}$$
and
$${\rm K}_0^+(A)=\{(n,2n)\in \mathbb{Z}\oplus \mathbb{Z}:\,
n\in \mathbb{N}\}\cong \mathbb{N}.$$

Note that  $A$ is a non-unital subhomogenous algebra of stable rank one, 
hence, $A$ has cancellation of projections.
Since $0$ is the unique projection in $A$, i.e., $A$ is projectionless, $A$ doesn't have a countable approximate unit consisting of projections. Then we can not lift the generator $(1,2)\in {\rm K}_0^+(A)$ to any projection in $A$. We have
$$\Sigma \,A=\{(0,0)\}\subset {\rm K}_0^+(A),$$
which means Definition \ref{scale def} (iii) is not satisfied.
Also note that the generator $(1,2)$ can be lifted to a projection $p$ in $M_2(A)$ with
$$
p(0)=\fontsize{8pt}{5pt}\selectfont\left(\!\!\begin{array}{cccc}
1& & &\\[1.5 mm]
&1 & &\\[1.5 mm]
& &0 &\\[1.5 mm]
& & &0
\end{array}\!\!\right)\quad {\rm and}\quad
p(1)=\fontsize{8pt}{5pt}\selectfont\left(\!\!\begin{array}{cccc}
1& & &\\[1.5 mm]
&0 & &\\[1.5 mm]
& &1 &\\[1.5 mm]
& & &0
\end{array}\!\!\right).
$$
\end{remark}


\section{Unital Absorbing Extensions}
In this section, we  consider the unital extensions of $A$ by $B$:
$$
0\to B\to E\to A\to 0,
$$
where $A,B,E$ are all separable ${\rm C}^*$-algebra of stable rank one and real rank zero, $B$ is stable and $A$ is unital simple. We will always identify $B$ as its image in $E$.


The following proposition is well-known, see \cite[Proposition 4]{LR}.
\begin{proposition}\label{lin inj}
Let
$$
0 \rightarrow B \rightarrow E \rightarrow A \rightarrow 0
$$
be an extension  of $\mathrm{C}^{*}$-algebras. Let $\delta_{j}: \mathrm{K}_{j}(A) \rightarrow \mathrm{K}_{1-j}(B)$ for $j=0,1$, be the index maps of the sequence.

(i) Assume that $A$ and $B$ have real rank zero. Then the following three conditions are equivalent:

(a) $\delta_{0} \equiv 0$,

(b) $rr(E)=0$,

(c) all projections in $A$ are images of projections in $E$.

(ii) Assume that $A$ and $B$ have stable rank one. Then the following are equivalent:

(a) $\delta_{1} \equiv 0$,

(b) $sr(E)=1$.

If, in addition, $E$ (and $A$) are unital, then (a) and (b) in (ii) are equivalent to

(c) all unitaries in $A$ are images of unitaries in $E$.

\end{proposition}
Then we have a direct corollary as follows which indicates that it is interesting to consider the extension with
trivial index maps.
\begin{corollary}\label{0 index sn}
Let $A,B$ be ${\rm C}^*$-algebras of stable rank one and real rank zero. Let $e\,:\,$$0\to B\to E\to A\to 0$ be an extension.
Then
$e$ has  trivial index maps if and only if $E$ has stable rank one and real rank zero.
\end{corollary}

Let $A, B$ be separable ${\rm C}^*$-algebras of stable rank one and  real rank zero.
The lattice
of ideals in $A$ is isomorphic to the lattice of order ideals in ${\rm K}_0(A)$ (see \cite{Z},
\cite{Goo2}). A positive morphism of groups $\phi : {\rm K}_0(A) \to {\rm K}_0(B)$ induces a map
$I\mapsto I'$ from the ideals of $A$ to the ideals of $B$. Here $I'$ is the unique ideal in
$B$ such that the image of ${\rm K}_0(I')$ in ${\rm K}_0(B)$ is the order ideal generated by $\phi({\rm K}_0(I))$.
The map ${\rm K}_0(I)\to {\rm K}_0(A)$ is injective, since $A$ has stable rank
one.
\begin{proposition} \label{full key}
Let $A$ be a separable ${\rm C}^*$-algebra of stable rank one and real rank zero and $p$ be a projection in $\mathcal{M}(A)$. Denote  $\iota:\, pAp\to A$
to be the natural embedding map. Then ${\rm K}_0(\iota):\, {\rm K}_0(pAp)\to {\rm K}_0(A)$ is injective.

Moreover, we have the follow statements are equivalent:

(i) $pAp$ is full in $A$;

(ii) ${\rm K}_0(\iota):\, {\rm K}_0(pAp)\to {\rm K}_0(A)$ is an isomorphism.
\end{proposition}
\begin{proof}
Consider the natural inclusion diagram
$$
\xymatrixcolsep{2pc}
\xymatrix{
{\,\,pAp\,\,} \ar[r]^-{\iota}\ar[d]_-{\iota_1}
& {\,\,A,\,\,}
 \\
{\,\,I_p\,\,} \ar[ur]_-{\iota_2}}
$$
where $I_p$ is the ideal of $A$ generated by $pAp$.
And it induces the commutative diagram as follows
$$
\xymatrixcolsep{2pc}
\xymatrix{
{\,\,{\rm K}_0(pAp)\,\,} \ar[r]^-{{\rm K}_0(\iota)}\ar[d]_-{{\rm K}_0(\iota_1)}
& {\,\,{\rm K}_0(A).\,\,}
 \\
{\,\,{\rm K}_0(I_p)\,\,} \ar[ur]_-{{\rm K}_0(\iota_2)}}
$$
By
Theorem \ref{herideal} and Remark \ref{iota inj}, we have
$pAp\otimes \mathcal{K}$ and $I_p\otimes \mathcal{K}$
are canonically isomorphic and
${\rm K}_0(\iota_1)$ is an isomorphism.
By Proposition \ref{lin inj}(ii)(a), we have ${\rm K}_0(\iota_2)$ is injective.
Then the composed map ${\rm K}_0(\iota)$ is injective.

For (i) $\Rightarrow$ (ii), $pAp$ is full in $A$ implies $I_p=A$,
and hence, ${\rm K}_0(\iota_2)$ is an isomorphism.
Thus, the composed map ${\rm K}_0(\iota)$ is an isomorphism.

For (ii) $\Rightarrow$ (i),  we have shown that ${\rm K}_0(\iota):\, {\rm K}_0(pAp)\to {\rm K}_0(A)$
is  an injective map. By assumption, ${\rm K}_0(\iota):\, {\rm K}_0(pAp)\to {\rm K}_0(A)$ is an isomorphism.
This is equal to say that the ideal generated by ${\rm K}_0(\iota)({\rm K}_0(pAp))$ in ${\rm K}_0(A)$ is ${\rm K}_0(A)$ itself.
Since $A$ a separable ${\rm C}^*$-algebra of stable rank one and real rank zero, we have ${\rm Lat}(A)\cong {\rm Lat}({\rm K}_0(A))$.
In addition, the above argument implies the ideal generated by $pAp$ in $A$ is $A$ itself.

\end{proof}

\begin{lemma}\label{f lemma}
Let $A,B$ be nuclear, separable ${\rm C}^*$-algebras of stable rank one and real rank zero. Suppose that $A$ is unital and $B$ is stable. 
Let $e$ be a unital essential  extension
$$
0\to B\to E\to A\to 0.
$$
For any projection $f$ in $M_k(E)$ with $1_E\leq f\leq  1_{M_k(E)}$
for some $k\in\mathbb{N}$, we have  ${f M_k(B)f}\cong B$ is  stable and $B$ is full in $f M_k(B)f$.
\end{lemma}

\begin{proof}
Consider the natural inclusion diagram
$$
\xymatrixcolsep{2pc}
\xymatrix{
{\,\,B\,\,} \ar[r]^-{\iota_3}\ar[d]^-{\iota_1}
& {\,\,M_k(B),\,\,}
 \\
{\,\,fM_k(B)f\,\,} \ar[ur]_-{\iota_2}}
$$
where $B$ is regarded as the first corner
of $M_k(B)$. It induces the following commutative diagram
$$
\xymatrixcolsep{2pc}
\xymatrix{
{\,\,{\rm K}_0(B)\,\,} \ar[r]^-{{\rm K}_0(\iota_3)}\ar[d]^-{{\rm K}_0(\iota_1)}
& {\,\,{\rm K}_0(M_k(B)).\,\,}
 \\
{\,\,{\rm K}_0(fM_k(B)f)\,\,} \ar[ur]_-{{\rm K}_0(\iota_2)}}
$$

Since $B$ is full in $M_k(B)$, then $fM_k(B)f$ is also full in $M_k(B)$. By
Theorem \ref{herideal} and Remark \ref{iota inj},  we have
both ${\rm K}_0(\iota_2)$ and ${\rm K}_0(\iota_3)$ are isomorphisms.
By commutativity, we have ${\rm K}_0(\iota_1)$ is also an isomorphism, by  Proposition \ref{full key},
we have $B$ is full in $fM_k(B)f$.







For any projection $p$ in $B$, as 
${\rm K}_0(\iota_1)([p])=[\iota_1(p)]$,
we have
$${\rm K}_0(\iota_2)\circ {\rm K}_0(\iota_1)(\Sigma \,B)\subset{\rm K}_0(\iota_2) \left(\Sigma\, {(fM_k(B)f)}\right)\subset \Sigma \,(M_k(B)).$$
Apply Proposition \ref{rodam p1} for $B$ and $M_k(B)$, then
$$
\Sigma\, B= {\rm K}_0^+(B)= \Sigma\,( M_k(B)).
$$
Since $B$ is full in $M_k(B)$, by
Theorem \ref{herideal} and Remark \ref{iota inj}, we have
$$
{\rm K}_0(\iota_2)\circ {\rm K}_0(\iota_1)(\Sigma\, B)=\Sigma\,( M_k(B)),$$ we obtain $${\rm K}_0(\iota_1)(\Sigma \,B)=\Sigma\,{(fM_k(B)f)},$$
that is, $$
\Sigma{(fM_k(B)f)}={\rm K}_0(\iota_1)({\rm K}_0^+ (B))={\rm K}_0^+ (fM_k(B)f).
$$
By Proposition \ref{rodam p1}, $fM_k(B)f$ is stable, and hence, by Theorem \ref{herideal},
$$ fM_k(B)f\cong B.$$

\end{proof}

\begin{lemma}\label{np lemma}
Let $A,B$ be nuclear, separable ${\rm C}^*$-algebras of stable rank one and real rank zero. Suppose that $A$ is unital,  $B$ is stable and $({\rm K}_0(B),{\rm K}_0^+(B))$  is weakly unperforated. Let $e$ be a unital essential extension with trivial index maps
$$
0\to B\to E\to A\to 0.
$$
Suppose $p,f$ are projections in $M_k(E)$ satisfying
$$
[1_E]\leq n\cdot [p]=[f],\,\,p\leq f,\,\, 1_E\leq f\leq 1_{M_k(E)}
$$ for some $n,k \in\mathbb{N}$. Then $ {pM_k(B)p}\cong
B$ is stable and  ${pM_k(B)p}$ is full in $fM_k(B)f$.

\end{lemma}
\begin{proof}
By Lemma \ref{f lemma}, we have
$f M_k(B)f\cong B$ and $B$ is full in $f M_k(B)f$.
Since $[f]=n\cdot [p]$, $1_E\leq f\leq 1_{M_k(E)}$ and $p\leq f$,
noting that $E$ has stable rank one,
we have $\underbrace{p\oplus p\oplus \cdots \oplus p}_{n-1}$ is Murray--von Neumann equivalent to $f-p$ in $M_l(E)$ for some $l$, set $p_1=p$, thus
we have
$$
f=p_1+p_2+\cdots+p_{n},
$$
where each $p_i\in M_k(E) $ is a subprojection of $f$, $p_i\sim p_1$  and $p_ip_j=0$  for any $i\neq j$.
Denote $v_{j}$ the partial isometry in  $M_k(E)$ with
$$
v_{j}v_{j}^*=p_1,\quad v_{j}^*v_{j}=p_j\quad {\rm ~~for~~ any}~~ j=1,2,\cdots,n.
$$
Now we show $pM_k(B)p$ is full in $f M_k(B)f$.
Note that $f M_k(B)f$ is also an ideal of $f M_k(E)f$, so we
only need to show that $f M_k(B)f$ is the ideal of $f M_k(E)f$
generated by $pM_k(B)p$. (If $C,D,W$ are ${\rm C}^*$-algebras with $D$ is an ideal of $W$ and
$C$ is an ideal of $D$, we will have $C$ is also an ideal of $W$.)

For any $x\in f M_k(B)f$,
then $x=\sum_{i=1}^n\sum_{j=1}^n p_i x p_j$.
For any $i,j$, $p_i x p_j\in f M_k(B)f$, we have $$v_i p_i x p_j v_j^*\in p M_k(B)p,$$
which implies
$$v_i^*v_i p_i x p_j v_j^*v_j=p_i x p_j,$$
i.e., $p_i x p_j$ is contained in the ideal of $f M_k(E)f$ generated by $p M_k(B)p$.
Then we have  $x$ is contained in the ideal of $f M_k(B)f$ generated by $p M_k(B)p$.
In general, we have $pM_k(B)p$ is full in $f M_k(B)f$.

Moreover, regarding $\{v_iv_j^*\}$ as the matrix units and $pM_k(B)p$
as the first corner, we have
$$ pM_k(B)p \otimes M_n\cong f M_k(B)f$$
naturally.

Since $pM_k(B)p$ is full in $f M_k(B)f$, by Theorem \ref{herideal}, we have
$$
pM_k(B)p\otimes \mathcal{K}\cong  fM_k(B)f\otimes \mathcal{K}\cong B.
$$
Then
$({\rm K}_0(pM_k(B)p),{\rm K}_0^+(pM_k(B)p))\cong({\rm K}_0(B),{\rm K}_0^+(B)) $
is weakly unperforated.

Denote the natural embedding map $\iota:\,pM_k(B)p \to fM_k(B)f$,
by Theorem \ref{herideal} and Remark \ref{iota inj}, we
identify ${\rm K}_0^+(pM_k(B)p)$ with ${\rm K}_0^+(fM_k(B)f)$ through ${\rm K}_0(\iota)$.

Apply Proposition \ref{rodam p1} for $fM_k(B)f$,  we obtain
$$
n\,\widehat{\cdot}\, \Sigma\, (pM_k(B)p)=\Sigma (fM_k(B)f)={\rm K}_0^+(fM_k(B)f)={\rm K}_0^+(pM_k(B)p).
$$
Then by Proposition \ref{rodam p2}, we have
$$\Sigma\, (pM_k(B)p)={\rm K}_0^+(pM_k(B)p).$$
By Proposition \ref{rodam p1}, $pM_k(B)p$ is stable,  and hence, $pM_k(B)p\cong B$.

\end{proof}

\begin{theorem}\label{main thm}
Let $A,B$ be nuclear, separable ${\rm C}^*$-algebras of stable rank one and real rank zero. Suppose that $A$ is unital simple,  $B$ is stable and $({\rm K}_0(B),{\rm K}_0^+(B))$  is weakly unperforated. Let $e$ be a unital extension with trivial index maps
$$
0\to B\rightarrow E\xrightarrow{\pi} A\to 0.
$$
Then $e$ is absorbing.
\end{theorem}
\begin{proof}
We only need  to prove $e$ is purely large.

For any $c\in E\backslash B$, $\pi(c)\neq 0$. Since $E$ has real rank zero, then there exists a nonzero projection $p\in \overline{cEc^*}$ such that $p\in E\backslash B$.
Otherwise, if such $p$ doesn't exist, by $rr(\overline{cEc^*})=0$, we will have $\overline{cEc^*}\subset B$,
which implies $c\in B$ and forms a contradiction.

Note that $pBp\subset \overline{c B c^*}$, we will show that $pBp$ is stable and is full in $B$.

Since $p\in E\backslash B$, $\pi(p)$ is a nonzero projection in $A$. As $A$ is unital simple and has stable rank one, there exists an integer $n$ such that $n\cdot [\pi(p)]\geq 2[1_A]$. Now we have
$$
{\rm K}_0(\pi)(n\cdot[p]-2[1_E])\geq 0\quad {\rm in}\,\,{\rm K}_0^+(A).
$$
As $E$  has real rank zero, by Proposition \ref{lin inj}(i)(c),
there exists $x\in {\rm K}_0^+(E)$ such that
$$
{\rm K}_0(\pi)(n\cdot[p]-2[1_E])={\rm K}_0(\pi)(x).
$$
Recall that $e$ has trivial index maps, then the following is an exact sequence
$$
0\to {\rm K}_0(B)\rightarrow {\rm K}_0(E)\xrightarrow{{\rm K}_0(\pi)}
{\rm K}_0(A)\to 0.
$$
So ${\rm K}_0(B)$ can be identified as the subgroup of ${\rm K}_0(E)$.
Then from the exactness, there exist projections $d,r$ in $B$ 
such that
$$n[p]-2[1_E]+[d]-[r]=x\geq 0 \quad {\rm in}\,\,{\rm K}_0^+(E).$$
Moreover, we have
$$n[p]+n[d]\geq 2[1_E] \quad {\rm in}\,\,{\rm K}_0^+(E).$$
Since $E$ has stable rank one,  we can lift $n[p]+n[d]$ to a projection $f$  in $M_k(E)$ ($k\geq 2$) with
$$
1_E\leq1_{M_2(E)}\leq f\leq 1_{M_k(E)}.
$$
 Since $p\leq 1_E$ and $d\leq 1_E$, we regard $p\oplus d$ as a
projection in $M_2(E)\subset M_k(E)$. Then we have
$$
p\oplus d\leq1_{M_2(E)}\leq f \quad{\rm in}\,\, M_k(E).
$$
Apply Lemma \ref{np lemma}, we have
$
(p\oplus d)M_k(B)(p\oplus d)
\cong B
$
and
$(p\oplus d)M_k(B)(p\oplus d)$ is full in $fM_k(B)f$.

Consider the following commutative diagram
$$
\xymatrixcolsep{2pc}
\xymatrix{
{\,\,(p\oplus d)M_k(B)(p\oplus d)\,\,} \ar[r]^-{\iota_6}\ar[d]^-{\iota_1}
& {\,\,M_k(E)\,\,}
& {\,\,E\,\,} \ar[l]_-{\iota_4}
 \\
{\,\,fM_k(B)f\,\,} \ar[r]_-{\iota_2}
& {\,\,M_k(B)\,\,}\ar[u]^-{\iota_5}
& {\,\,B\,\,}\ar[l]^-{\iota_3} \ar[u]_-{\iota_7},}
$$
where for each $i$, $\iota_i$ is  the natural embedding map ($B$ and $E$ are regarded as the first corners
of $M_k(B)$ and $M_k(E)$, respectively).
Then it induces a  commutative diagram as follows
$$
\xymatrixcolsep{2pc}
\xymatrix{
{\,\,{\rm K}_0((p\oplus d)M_k(B)(p\oplus d))\,\,} \ar[r]^-{{\rm K}_0(\iota_6)}\ar[d]^-{{\rm K}_0(\iota_1)}
& {\,\,{\rm K}_0(M_k(E))\,\,}
& {\,\,{\rm K}_0(E)\,\,} \ar[l]_-{{\rm K}_0(\iota_4)}
 \\
{\,\,{\rm K}_0(fM_k(B)f)\,\,} \ar[r]_-{{\rm K}_0(\iota_2)}
& {\,\,{\rm K}_0(M_k(B))\,\,}\ar[u]^-{{\rm K}_0(\iota_5)}
& {\,\,{\rm K}_0(B)\,\,} \ar[l]^-{{\rm K}_0(\iota_3)} \ar[u]_-{{\rm K}_0(\iota_7)}.}
$$

Since $E$ has stable rank one and real rank zero, by Proposition \ref{full key}, on one hand, for each $1\leq i\leq7$, we have
${\rm K}_0(\iota_i)$ is injective;
 on the other hand, by the fullness we have obtained, for each $1\leq i\leq4$,
 we have ${\rm K}_0(\iota_i)$ is an isomorphism.

 Note that $(p\oplus d)M_k(B)(p\oplus d)$ is stable and also has stable rank one,
then for any projection $q\in B$, we can lift
$${\rm K}_0(\iota_1)^{-1}\circ {\rm K}_0(\iota_2)^{-1}\circ {\rm K}_0(\iota_3)([q])\in {\rm K}_0^+((p\oplus d)M_k(B)(p\oplus d))$$
to a projection $q'\in (p\oplus d)M_k(B)(p\oplus d)$. Then
$$
\iota_6(q')\leq p\oplus d \quad{\rm in}\,\,M_k(E).
$$
Combining with the commutativity, we have
$$
{\rm K}_0(\iota_5)\circ {\rm K}_0(\iota_3)([q])= {\rm K}_0(\iota_6)([q']) \leq [p\oplus d]=[p]+[d]
\quad{\rm in}\,\,{\rm K}_0^+(M_k(E)).
$$
From the arbitrariness of $q$, by choosing  $\widetilde{q}$ with $[\widetilde{q}]=[q]+[d]$,
we have
$$ {\rm K}_0(\iota_5)\circ {\rm K}_0(\iota_3)([q])\leq [p]\in {\rm K}_0^+(M_k(E)).\qquad(*)$$
Note that we can identify
${\rm K}_0(B)$ as the subgroup of ${\rm K}_0(E)$ through the injective map ${\rm K}_0(\iota_5)\circ {\rm K}_0(\iota_3)$.

As $p\leq1_E$, $pM_k(B)p$ is naturally identified as $pBp$ in $B$.
Denote $\iota:\,pBp\to B$ the natural inclusion.
Since $B$ has stable rank one, we can identify ${\rm K}_0^+(pBp)$ as a sub-semigroup of ${\rm K}_0^+(B)$ (Proposition \ref{full key}).
Now $(*)$ implies $\Sigma\,(pBp)={\rm K}_0^+(B)$, which means
${\rm K}_0^+(pBp)={\rm K}_0^+(B)$.
Then by Proposition \ref{rodam p1}, we have $pBp$ is stable, and  hence, $pBp\cong B.$
From Proposition \ref{full key}, we get $pBp$ is full in $B$. Now $e$ is purely large, then by Theorem \ref{ek thm}, $e$ is absorbing.

\end{proof}
\begin{corollary}\label{dasheying}
  Suppose $A,B,E$  satisfy the assumption of Theorem \ref{main thm}.
  Then for any $p\in \mathcal{P}(E\backslash B)$ and $q\in \mathcal{P} (B)$, we have $[q]\leq [p]$ in ${\rm K}_0(E)$.
\end{corollary}


\section{Applications}

\begin{definition}\rm
Let
$$e_i:\quad 0 \to B \to E_i \to A \to 0 $$
be two
extensions with Busby invariants $\tau_i$ for $i = 1, 2$.

Two unital extensions $e_1$ and $e_2$ are called $stably$ $strongly$ $unitarily$ $equivalent$, denoted by $e_1
\sim_{ss} e_2$, if there exist unital trivial extensions $\sigma_1,\sigma_2$  such that $e_1\oplus\sigma_1 \sim_s e_2\oplus\sigma_2$.

If $A$ is unital,  denote  $\mathrm{Ext}^u_{ss}(A, B)$ the set of stably strongly unitary equivalence classes of unital extensions of $A$ by $B$.
\end{definition}
From our main result Theorem \ref{main thm}, we have
\begin{corollary}\label{ss s}
Let $A,B$ be nuclear, separable ${\rm C}^*$-algebras of stable rank one and real rank zero.
 Suppose that $A$ is unital simple,  $B$ is stable and $({\rm K}_0(B),{\rm K}_0^+(B))$  is weakly unperforated.
Let $e_1, e_2$ be two unital extensions with trivial index maps:
$$
e_i\,:\,0\to B\to E_i\to A\to 0.
$$
 Then $e_1\sim_{ss}e_2 $ if and only if $e_1\sim_s e_2$.
\end{corollary}
\begin{definition}\rm
Given abelian groups $H,K$ and two extensions
$$
e_i\,\,:\,\,0\to K\to G_i\to H\to 0,\quad i=1,2,
$$
we say $e_1$, $e_2$ are $equivalent$, if there is a homomorphism  $\alpha$ making the following diagram commute:
$$
\xymatrixcolsep{2pc}
\xymatrix{
{\,\,0\,\,} \ar[r]^-{}
& {\,\,K\,\,} \ar[d]_-{id} \ar[r]^-{}
& {\,\,G_1\,\,} \ar[d]_-{\alpha} \ar[r]^-{}
& {\,\,H\,\,} \ar[d]_-{id} \ar[r]^-{}
& {\,\,0\,\,} \\
{\,\,0\,\,} \ar[r]^-{}
& {\,\,K\,\,} \ar[r]_-{}
& {\,\,G_2 \,\,} \ar[r]_-{}
& {\,\,H \,\,} \ar[r]_-{}
& {\,\,0\,\,}.}
$$
Denote $\mathrm{Ext}(H,K)$ the set of all the equivalent classes of extensions of $H$ by $K$. It is well-known that $\mathrm{Ext}(H,K)$ forms an abelian group.

Let  $h_0 \in H$. We consider the following
extension of $H$ by $K$ with base point $h_0$,
$$
e:\quad 0\to K\to (G, g_0)\xrightarrow{\psi} (H, h_0)\to 0,
$$
where $g_0 \in G$ and $\psi(g_0) = h_0$.

Suppose we have extensions
$$e_i:\quad
0\to K\to (G_i, g_i)\xrightarrow{\psi} (H, h_0)\to 0,\quad i = 1, 2$$
 with $\psi(g_i) = h_0$. We say  $e_1$ and $e_2$ are equivalent if there is a
homomorphism $\phi : G_1 \to G_2$ with $\phi(g_1) = g_2$ such that the following diagram commutes
$$
\xymatrixcolsep{2pc}
\xymatrix{
{\,\,0\,\,} \ar[r]^-{}
& {\,\,K\,\,} \ar[d]_-{id} \ar[r]^-{}
& {\,\,(G_1, g_1)\,\,} \ar[d]_-{\phi} \ar[r]^-{}
& {\,\,(H, h_0)\,\,} \ar[d]_-{id} \ar[r]^-{}
& {\,\,0\,\,} \\
{\,\,0\,\,} \ar[r]^-{}
& {\,\,K\,\,} \ar[r]_-{}
& {\,\,(G_2, g_2) \,\,} \ar[r]_-{}
& {\,\,(H, h_0) \,\,} \ar[r]_-{}
& {\,\,0\,\,}.}
$$

Let $\mathrm{Ext}((H, h_0), K)$ be the set of equivalent classes of all extensions of $H$ by $K$ with base point $h_0$.

\end{definition}

Let $A$ and $B$ be ${\rm C}^*$-algebras. When $A$ is unital, we set
$$\mathrm{Ext}_{[1]}(\mathrm{K}_*(A), \mathrm{K}_*(B)) = \mathrm{Ext}((\mathrm{K}_0(A), [1_A]), \mathrm{K}_0(B)) \oplus \mathrm{Ext}(\mathrm{K}_1(A), \mathrm{K}_1(B)),$$
$${\rm Hom}_{[1]}(\mathrm{K}_0(A), \mathrm{K}_1(B)) = \{\rho \in {\rm Hom}((\mathrm{K}_0(A), [1_A]), \mathrm{K}_1(B)) : \rho([1_A]) = 0\},$$
$${\rm Hom}_{[1]}(\mathrm{K}_*(A), \mathrm{K}_*(B)) = {\rm Hom}_{[1]}(\mathrm{K}_0(A), \mathrm{K}_1(B)) \oplus {\rm Hom}(\mathrm{K}_1(A), \mathrm{K}_0(B)).$$

Denote $\mathcal{N}$ the ``bootstrap" class of ${\rm C}^*$-algebras in \cite{RS}.
For the  unital version of ``UCT", Wei used the stably strongly unitary equivalence classes of unital extensions to get the following theorem (a similar statement can be found in \cite[Proposition 3.2]{X}).
\begin{theorem}{\rm (}\cite[Theorem 4.9]{W}{\rm )}\label{strong wei}
Suppose that $A$ is a unital separable nuclear ${\rm C}^*$-algebra with $A\in\mathcal{N}$ and $B$ is separable stable ${\rm C}^*$-algebra. Then
there is a short exact sequence of groups
$$
0 \to \mathrm{Ext}_{[1]}(\mathrm{K}_*(A), \mathrm{K}_*(B)) \to \mathrm{Ext}_{ss}^u(A, B) \to {\rm Hom}_{[1]}(\mathrm{K}_*(A), \mathrm{K}_*(B)) \to 0.
$$
\end{theorem}

Combining with Corollary \ref{ss s}, we get the following result.

\begin{theorem}\label{text}
Let $A,B$ be nuclear, separable ${\rm C}^*$-algebras of stable rank one and real rank zero with $A\in \mathcal{N}$. Assume that  $A$ is unital simple,  $B$ is stable and $({\rm K}_0(B),{\rm K}_0^+(B))$  is weakly unperforated. Denote ${\rm Text}_{s}^u(A,B)$  to be the set of strongly unitary equivalence classes of all the unital extensions of $A$ by $B$ with trivial index maps. Then  we have
$$
\mathrm{Ext}_{[1]}(\mathrm{K}_*(A), \mathrm{K}_*(B))\cong {\rm Text}_{s}^u(A,B).
$$
\end{theorem}
In \cite{AL}, Corollary \ref{dasheying} and Theorem \ref{text} are applied to construct a counterexample for Elliott Conjecture with real rank zero and obtain classification theorems of ${\rm C}^*$-algebras in the setting of stable rank one and real rank zero. 

\section*{Acknowledgements}
The authors are very grateful to  Yifeng Xue and Changguo Wei for many helpful discussions.  The first author was supported by NNSF of China (No.:12101113).
The second author was supported by NNSF of China (No.:12101102).

\end{document}